\documentclass[leqno]{amsart}  
\usepackage{amsmath,amstext,amsthm,amssymb,amsxtra}
\usepackage{todonotes}

\usepackage[greek,english]{babel}
\usepackage[utf8x]{inputenc}

\usepackage{txfonts,pxfonts,tikz} %also txfonts 
\usepackage[T1]{fontenc}

\usepackage{mathtools}
\mathtoolsset{showonlyrefs,showmanualtags} 

\usepackage[colorlinks,citecolor=red,pagebackref,hypertexnames=true]{hyperref}
\usepackage[backrefs,msc-links,nobysame]{amsrefs}

\theoremstyle{plain} % definition 
\newtheorem{lemma}{Lemma}[section] 
\newtheorem{proposition}{Proposition}[section] 
\newtheorem{theorem}{Theorem}[section] 
\newtheorem{corollary}{Corollary}[section] 

\theoremstyle{definition}

\theoremstyle{remark}
\newtheorem{remark}{Remark}[section]

\numberwithin{equation}{section}

\def\norm#1.#2.{\lVert#1\rVert_{#2}}
\def\Norm#1.#2.{\bigl\lVert#1\bigr\rVert_{#2}}
\def\NOrm#1.#2.{\Bigl\lVert#1\Bigr\rVert_{#2}}
\def\NORm#1.#2.{\biggl\lVert#1\biggr\rVert_{#2}}
\def\NORM#1.#2.{\Biggl\lVert#1\Biggr\rVert_{#2}}

\def\ip#1,#2,{\langle #1,#2\rangle}
\def\Ip#1,#2,{\bigl\langle#1,#2\bigr\rangle}
\def\IP#1,#2,{\Bigl\langle#1,#2\Bigr\rangle}

\def\Abs#1{\bigl\lvert#1\bigr\rvert}
\def\ABs#1{\biggl\lvert#1\biggr\rvert}

\def\XXint#1#2#3{{\setbox0=\hbox{$#1{#2#3}{\int}$}
     \vcenter{\hbox{$#2#3$}}\kern-.5\wd0}}

\def\eqdef{\stackrel{\mathrm{def}}{{}={}}}

\def\c{\textnormal {circle}}

\begin{document}
%%%%%%%%%%%%%%%%%%%%%%%%%%%%%  Title
\title[Circle Discrepancy for checkerboard measures] {Circle discrepancy for checkerboard measures}

\author[M. Kolountzakis]{{Mihail N. Kolountzakis}}
\address{M.K.: Department of Mathematics, University of Crete, Knossos Ave., GR-714 09, Iraklio, Greece}
\email{kolount@gmail.com}
\thanks{M.K.: Supported by research grant No 3223 from the Univ. of Crete.}

\author[I. Parissis]{Ioannis Parissis}
\address{I.P.: Department of Mathematics and Statistics, P.O.B. 68 (Gustaf H\"allstr\"omin katu 2b), FI-00014, University of Helsinki, Finland.}
\email{ioannis.parissis@gmail.com}
\thanks{I.P.: Research partially supported by CAMGSD-LARSYS through  Funda\c{c}\~{a}o para a Ci\^{e}ncia e Tecnologia (FCT/Portugal), program POCTI/FEDER and partially by the European Union through the ERC Starting Grant `Analytic–probabilistic methods for borderline singular integrals'.}

\begin{abstract} Consider the plane as a union of congruent unit squares in a checkerboard pattern, each square colored black or white in an arbitrary manner. The discrepancy of a curve with respect to a given coloring is the difference of its white length minus its black length, in absolute value. We show that for every radius $t\geq 1$ there exists a full circle of radius either $t$ or $2t$ with discrepancy greater than $c\sqrt{t}$ for some numerical constant $c>0$. We also show that for every $t\geq 1$ there exists a circular arc of radius exactly $t$ with discrepancy greater than $c\sqrt{t}$. Finally we investigate the corresponding problem for more general curves and their interiors. These results answer questions posed by Kolountzakis and Iosevich.
\end{abstract}

\maketitle

%%%%%%%%%%%%%%%%%%%%%%%%%%%%%% SECTION  SECTION SECTION
%%%%%%%%%%%%%%%%%%%%%%%%%%%%%% SECTION  SECTION SECTION 
\section{Introduction.} \label{s.intro}
In this note we take up the investigation, initiated in \cite{K} and continued in \cite{IK}, concerning the \emph{discrepancy} of various geometrical shapes with respect to non-atomic measures (colorings). In order to discuss the problems we are interested in we need to introduce some notation. As in \cite{K,IK} we divide the Euclidean plane $\mathbb R^2$ into the unit cells 
$$Q_p\eqdef[p_1,p_1+1)\times[p_2,p_2+1),\quad p=(p_1,p_2)\in\mathbb Z^2,$$
and color each one of the cells either black or white. Thus a \emph{checkerboard coloring} $f$ of the plane is a function
$$f:\mathbb R^2\to \{-1,+1\},$$
such that $f$ is constant on each unit cell $Q_p$. Now let $S$ be a simple curve lying in the checkerboard-plane and $f$ be a coloring as before. We define the \emph{discrepancy} of $S$ with respect to the given coloring $f$ to be the difference of the `white' length of $S$ against the `black' length of $S$, in absolute value. In \cite{K} it was proved that for any checkerboard coloring there exist arbitrarily long line segments $I$ with discrepancy at least $c\sqrt{|I|}$, for some numerical constant $c>0$. On the other hand in \cite{IK} the authors proved that for arbitrarily large $R>0$ there exists a circular arc of radius comparable to $R$ which has discrepancy at least $c\sqrt{R}$ for some numerical constant $c>0$. The authors in \cite{IK} also ask whether there is a \emph{full circle} $C$ with large discrepancy. We answer this question in a strong form by showing that for every radius $t\geq 1$ there exists a full circle of radius either $t$ or $2t$ with discrepancy at least $c\sqrt t$. Noting by $C(x,t)$ the circle of radius $x\in\mathbb R^2$ and radius $t>0$ we have:

%%%%%%%%%%%%%%%%%%%%%%%%%%%%%% THEOREM THEOREM THEOREM
%%%%%%%%%%%%%%%%%%%%%%%%%%%%%% THEOREM THEOREM THEOREM
\begin{theorem}\label{t.fullcircle}
	Let $f$ be a checkerboard coloring of the plane as before and let $t\geq 1$.There exists a $x\in\mathbb R^2$ such that
$$\mbox{either}\quad \ABs{\int_{C(x,t)} f} \geq c t^\frac{1}{2}\quad\mbox{or}\quad \ABs{\int_{C(x,2t)} f} \geq c (2t)^\frac{1}{2},$$
for some numerical constant $c>0$.
\end{theorem}
%%%%%%%%%%%%%%%%%%%%%%%%%%%%%% THEOREM THEOREM THEOREM
%%%%%%%%%%%%%%%%%%%%%%%%%%%%%% THEOREM THEOREM THEOREM

We also show that if we just care about finding \emph{arcs} with large discrepancy, then we can do so for any fixed radius $t\geq 1$.

%%%%%%%%%%%%%%%%%%%%%%%%%%%%%% THEOREM THEOREM THEOREM
%%%%%%%%%%%%%%%%%%%%%%%%%%%%%% THEOREM THEOREM THEOREM
\begin{theorem}\label{t.singleradius}
	Let $f$ be a checkerboard coloring of the plane as before and let $t\geq 1$.There exists a circular arc $K$ of radius $t$ such that
$$ \ABs{\int_K f} \geq c t^\frac{1}{2},$$
for some numerical constant $c>0$.
\end{theorem}
%%%%%%%%%%%%%%%%%%%%%%%%%%%%%% THEOREM THEOREM THEOREM
%%%%%%%%%%%%%%%%%%%%%%%%%%%%%% THEOREM THEOREM THEOREM

The results in \cite{K}, \cite{IK} as well as Theorem \ref{t.fullcircle} and Theorem \ref{t.singleradius}, are direct consequences of their finite counterparts. To make this precise, let $N$ be a positive integer and write $Q_N$ for the square $Q_N\eqdef[0,N)^2$. We now consider $Q_N$ as a union of congruent unit cells in the form 
\begin{align}
Q(p)\eqdef p+[0,1)^2, \quad p=(p_1,p_2)\in G,
\end{align}
where $G$ is the part of the lattice $\mathbb Z^2$ that lies in $Q_N$, that is $G\eqdef\{(p_1,p_2):0\leq p_1,p_2\leq N-1\}$. A \emph{coloring} of $Q_N$ will be a function of the form 
\begin{align}
	f_N:Q_N\to \{-1,+1\},\quad f_N\mbox{ constant in each cell }Q(p).
\end{align}
We extend $f_N$ to the whole plane $\mathbb R^2$ by setting $f_N\equiv 0$ outside $Q_N$. The discrepancy of a circle $C(x,t)$ with center $x\in\mathbb R^2$ and radius $t>0$ is defined as
$$ D_t(f_N,x)\eqdef \int_{C(x,t)} f_N=(f_N*\sigma_t)(x),$$
where $\sigma_t$ is the arc-length measure on a circle of center $0$ and radius $t$. A problem that arises is that discrepancy of circles with respect to a finite coloring in general only corresponds to discrepancy of arcs with respect to a coloring of the whole plane. The reason of course is that a circle $C(x,t)$ might intersect $Q_N$, and even have large discrepancy with respect to the finite coloring of $Q_N$, without necessarily lying entirely inside $Q_N$.

For example, Iosevich and the first author prove in \cite[ Theorem 1]{IK} that for any coloring $f_N$ of $Q_N$, there exists a circular arc $K$ of radius $R$, $N/5<R<N/4$, with
	$$\ABs{\int_{K} f_N}\geq c N^\frac{1}{2},$$
for some numerical constant $c>0$. The authors are not able to conclude that there is a full circle with large discrepancy since their main tool is to show that the $L^2$-type discrepancy
$$\frac{1}{N^3}\int_{N/5} ^{N/4} \int_{\mathbb R^2}|D_t(f_N,x)|^2 dx\, dt,$$
is large. However, the previous $L^2$ integral takes into account arcs as well as full circles. Furthermore, the averaging in the radial variable results to circles or circular arcs of radius \emph{comparable} to $N$ instead of radius exactly $N$. 

In this note we partially fix the previous two problems by avoiding the radial averaging. We also show that circles that do not lie entirely inside $Q_N$ do not significantly contribute to the $L^2$ norm $\norm D_t(f,\cdot).L^2. ^2+\norm D_{2t}(f,\cdot).L^2. ^2$ when $N\gtrsim t^2$. This results to a full circle of radius either $t$ or $2t$ with large discrepancy if $t$ is small comparable to $N$.

Theorem \ref{t.fullcircle} is an immediate consequence of the following theorem:

%%%%%%%%%%%%%%%%%%%%%%%%%%%%%% THEOREM THEOREM THEOREM
%%%%%%%%%%%%%%%%%%%%%%%%%%%%%% THEOREM THEOREM THEOREM
\begin{theorem}\label{t.fullcircleN} Let $t\geq 1$ and for a positive integer $N\geq 100 t^2$ consider any finite coloring $f_N : Q_N\to\{−1, +1\}$ of $Q_N$. There exists $x\in\mathbb R$ such that the circle $C(x,2t)\subset Q_N$ and
$$\mbox{either}\quad\ABs { \int_{C(x,t)} f_N } \geq c t^\frac{1}{2} \quad\mbox{or}\quad \ABs { \int_{C(x,2t)} f_N } \geq c  (2t)^\frac{1}{2} ,$$
where $c>0$ is some numerical constant.	
\end{theorem}
%%%%%%%%%%%%%%%%%%%%%%%%%%%%%% THEOREM THEOREM THEOREM
%%%%%%%%%%%%%%%%%%%%%%%%%%%%%% THEOREM THEOREM THEOREM

Similarly, Theorem \ref{t.singleradius} is a consequence of:
%%%%%%%%%%%%%%%%%%%%%%%%%%%%%% THEOREM THEOREM THEOREM
%%%%%%%%%%%%%%%%%%%%%%%%%%%%%% THEOREM THEOREM THEOREM
\begin{theorem}\label{t.singleradiusN} Let $f_N : Q_N\to\{−1, +1\}$ be a finite coloring of $Q_N$ and $N\simeq t$. There exists a circle $C$ of radius $t$ such that
	$$\ABs{\int_C f_N }\geq c\sqrt{t},$$
	where $c>0$ is some numerical constant.
\end{theorem}
%%%%%%%%%%%%%%%%%%%%%%%%%%%%%% THEOREM THEOREM THEOREM
%%%%%%%%%%%%%%%%%%%%%%%%%%%%%% THEOREM THEOREM THEOREM

%%%%%%%%%%%%%%%%%%%%%%%%%%%%%% REMARK REMARK REMARK
%%%%%%%%%%%%%%%%%%%%%%%%%%%%%% REMARK REMARK REMARK
\begin{remark}
	Note that in Theorem \ref{t.singleradiusN} we cannot guarantee that the circle $C$ is contained in $Q_N$. Thus, Theorem \ref{t.singleradiusN} only results to an \emph{arc} of radius $t$ in the infinite coloring of the plane with discrepancy $\sim \sqrt{t}$.
\end{remark}
%%%%%%%%%%%%%%%%%%%%%%%%%%%%%% REMARK REMARK REMARK
%%%%%%%%%%%%%%%%%%%%%%%%%%%%%% REMARK REMARK REMARK

We note that discrepancies with respect to non-atomic colorings have been considered by Rogers in \cite{RO1}, \cite{RO2} and \cite{RO3} where the author considers, among other things, the discrepancy of lines and half spaces with respect to finite colorings of the plane. Rogers proves lower bounds for the discrepancy of these families of sets with respect to generalized colorings. His results do not seem to be comparable to the results in this paper.

The rest of the paper is organized as follows. In Section \ref{s.fullcircles} we use the classical asymptotic estimates for the Fourier transform of the arc-length measure on the circle in order to prove Theorem \ref{t.fullcircleN}. In Section \ref{s.singleradius} we prove Theorem \ref{t.singleradiusN} by an appeal to the asymptotic estimates of the Fourier transform of the arc-length measure together with an appropriate Poincar\'e-type inequality. Finally in Section \ref{s.generalsets} we discuss the discrepancy of  more general families of sets with respect to a coloring of the plane. The corresponding lower bounds are contained in Theorem \ref{t.segment}. The main tool for these estimates are lower bounds for the averages of Fourier transforms of indicator functions. For the sake of completeness, we include these estimates and their proofs  in Section \ref{s.lowerfourier}.

%%%%%%%%%%%%%%%%%%%%%%%%%%%%%% SECTION  SECTION SECTION
%%%%%%%%%%%%%%%%%%%%%%%%%%%%%% SECTION  SECTION SECTION 
\section{Notations} Throughout the paper $c$ denotes a numerical positive constant which might change even in the same line of text. We often suppress numerical constants by using the symbol $\lesssim$. Thus $A\lesssim B$ means that $A\leq c B$ for $c$ as described. Likewise the notation $A\simeq B$ means that $A\lesssim B$ and $A\gtrsim B$. We write $B(x,r)$ for the Euclidean disk of radius $r>0$ centered at $x\in\mathbb R^2$. We also write $C(x,r)=\partial B(x,r)$ for the circle of radius $r>0$, centered at $x\in\mathbb R^2$. For the unit circle of $\mathbb R^2$ we also use the symbol $S^{1}=C(0,1)$.

%%%%%%%%%%%%%%%%%%%%%%%%%%%%%% SECTION  SECTION SECTION
%%%%%%%%%%%%%%%%%%%%%%%%%%%%%% SECTION  SECTION SECTION
\section{Full circles of large discrepancy}\label{s.fullcircles} Recall that the discrepancy of a circle $C(x,t)$ with respect to the coloring $f_N$ of the square $Q_N=[0,N)^2$ is defined as 
$$D_t(f_N)(x)\eqdef \int_{C(x,t)}f= (f*\sigma_t)(x),$$
where $\sigma_t$ is the arc-length measure on the circle $C(0,t)$. Observe that the function $(f_N *\sigma_t)(x)$ has support in $Q_N+B(0,t)$ in general. However in Theorems \ref{t.fullcircleN} and \ref{t.singleradiusN} we only need to consider values $t\lesssim N$ so the measure of the support is comparable to $N^2$. We thus study the $L^2$ discrepancy 
$$D_t(f_N,2)\eqdef \bigg(\frac{1}{N^2}\int_{\mathbb R^2} |(f_N*\sigma_t)(x)|^2 dx \bigg)^\frac{1}{2},$$
since we obviously have the bound
$$\sup_{x\in\mathbb R^2}| D_t(f_N)(x) | \gtrsim D_t(f_N,2).$$
Furthermore, denoting by $\hat \sigma_1$ the Fourier transform of the measure $d\sigma_1$,
$$\hat \sigma_1 (\xi)=\int_{S^1}e^{-2\pi i x'\cdot\xi}d\sigma_1(x'),$$
we have that
$$D_t(f_N,2)^2=\frac{t^2}{N^2}\int_{\mathbb R^2} |\widehat{ f_N}(\xi)|^2 |\hat \sigma_1 (t \xi)|^2 d\xi.$$
The following Lemma is the most essential part of the proof of Theorem \ref{t.fullcircleN}.

%%%%%%%%%%%%%%%%%%%%%%%%%%%%%% LEMMA LEMMA LEMMA
%%%%%%%%%%%%%%%%%%%%%%%%%%%%%% LEMMA LEMMA LEMMA
\begin{lemma}\label{l.double} For all $|\xi|\geq \frac{1}{2\pi}$ we have that
	$$|\hat \sigma _1 (\xi)|^2+|\hat \sigma_1 (2\xi)|^2\gtrsim\frac{1}{|\xi|}.$$
\end{lemma}
%%%%%%%%%%%%%%%%%%%%%%%%%%%%%% LEMMA LEMMA LEMMA
%%%%%%%%%%%%%%%%%%%%%%%%%%%%%% LEMMA LEMMA LEMMA

%%%%%%%%%%%%%%%%%%%%%%%%%%%%%% PROOF PROOF PROOF
%%%%%%%%%%%%%%%%%%%%%%%%%%%%%% PROOF PROOF PROOF
\begin{proof} Setting $|\xi|=r$ we express the radial function $\hat \sigma_1$ by the well known formula
	 $$\hat \sigma_1 (r)=2\pi J_0(2\pi r) ,$$
where $J_0$ is the $0$-th order Bessel function. We use the asymptotic estimate
$$J_0(r)\simeq \frac{1}{\sqrt{ r}}\bigg(\cos(r-\frac{\pi}{4})+e(r) \bigg),$$
where the error term satisfies
$$|e(r)|\leq  \frac{1}{5r},$$
for $r\geq 1$. This is classical as $r\to+\infty$ but with a little more effort one can get the validity of the previous estimate for all $r\geq 1$. The previous asymptotic estimate easily implies that
\begin{align*}
	|\hat \sigma_1(r)|^2+|\hat \sigma_1(2r)|^2\gtrsim\frac{1}{r},
\end{align*}
for all $r\geq\frac{7}{2\pi}$.
For $\frac{1}{2\pi}\leq r \leq \frac{7}{2\pi}$ one can just directly check the zeros of $J_0$ to see that there is no $r$ so that $J_o(2\pi r)=J_0(4\pi r)=0$. We refer the interested reader to \cite[p. 113, \S 6.3]{M} where an identical argument is used for the derivation of a formula involving the $1$-st order Bessel function. 
\end{proof}
%%%%%%%%%%%%%%%%%%%%%%%%%%%%%% PROOF PROOF PROOF
%%%%%%%%%%%%%%%%%%%%%%%%%%%%%% PROOF PROOF PROOF

%%%%%%%%%%%%%%%%%%%%%%%%%%%%%% COROLLARY COROLLARY COROLLARY
%%%%%%%%%%%%%%%%%%%%%%%%%%%%%% COROLLARY COROLLARY COROLLARY
\begin{corollary}\label{c.tor2t} For any $t\geq 1$ and any positive integer $N$ we have that
	$$D_t(f_N,2)^2+D_{2t}(f_N,2)^2\gtrsim t.$$
\end{corollary}
%%%%%%%%%%%%%%%%%%%%%%%%%%%%%% COROLLARY COROLLARY COROLLARY
%%%%%%%%%%%%%%%%%%%%%%%%%%%%%% COROLLARY COROLLARY COROLLARY

%%%%%%%%%%%%%%%%%%%%%%%%%%%%%% PROOF PROOF PROOF
%%%%%%%%%%%%%%%%%%%%%%%%%%%%%% PROOF PROOF PROOF
\begin{proof} We use Plancherel's theorem to write
	\begin{align*}
		D_t(f_N,2)^2+D_{2t}(f_N,2)^2 & =\frac{1}{N^2}\int_{\mathbb R^2} |\widehat{ f_N}(\xi)|^2 \big(|\hat \sigma_t(\xi)|^2+|\hat \sigma_{2t}(\xi)|^2 \big)\ d\xi \\ 
& \gtrsim \frac{1}{N^2}\int_{|\xi|\leq \frac{1}{2\pi}} |\widehat {f_N} (\xi/t)|^2 \big(|\hat \sigma_1 ( \xi)|^2+|\hat \sigma_ 1 (2 \xi)|^2 \big)\ d\xi \\ & \quad +\frac{1}{N^2}\int_{|\xi|> \frac{1}{2\pi}} |\widehat {f_N} (\xi/t)|^2 \big(|\hat \sigma_1 ( \xi)|^2+|\hat \sigma_ 1 (2 \xi)|^2 \big)\ d\xi\eqdef I+II.
	\end{align*}
For $I$ observe that $J_0 (2\pi \cdot)$  has no root in the range $|\xi|\leq \frac{1}{2\pi}$. We immediately get
$$ |I|\gtrsim \frac{1}{N^2}\int_{|\xi|\leq \frac{1}{2\pi}}|\widehat {f_N}(\xi/t)|^2 d\xi \geq \frac{t^2}{N^2}\int_{|\xi|\leq \frac{1}{2\pi t}}|\widehat{ f_N}(\xi)|^2 d\xi .$$	
For $II$ we use Lemma \ref{l.double} to write 
$$|II| \gtrsim \frac{1}{N^2}\int_{|\xi|>\frac{1}{2\pi}} |\widehat {f_N}(\xi/t)|^2 \frac{1}{|\xi|}d\xi \gtrsim\frac{t}{N^2}\int_{ [-\frac{1}{2},\frac{1}{2}]^2 \setminus \{\ |\xi|>\frac{1}{2\pi t } \ \}} |\widehat {f_N}(\xi)|^2 d\xi.$$
Combining the estimates and remembering that $t\geq 1$ we get 
$$ 	D_t(f_N,2)^2+D_{2t}(f_N,2)^2 \gtrsim \frac{t}{N^2}\int_{[-\frac{1}{2},\frac{1}{2}]^2} |\widehat {f_N}(\xi)|^2d \xi \gtrsim t $$
where we have used that
\begin{align}\label{e.1period}
	|\widehat {f_N}(\xi)|^2=\bigg|\frac{\sin(\pi\xi_1)}{\pi\xi_1}\frac{\sin(\pi\xi_2)}{\pi\xi_2}\sum_{j,k=0} ^{N-1}z_{jk} e^{2\pi i (j\xi_1+k\xi_2)}\bigg|^2\gtrsim \bigg|\sum_{j,k=0} ^{N-1}z_{jk} e^{2\pi i (j\xi_1+k\xi_2)}\bigg|^2
	\end{align}
for $\xi\in[-\frac{1}{2},\frac{1}{2}]^2$. This in turn is a consequence of the elementary estimate $|\sin(\pi x)|\geq 2|x|$ for $|x|\leq \frac{1}{2}$.
\end{proof}
%%%%%%%%%%%%%%%%%%%%%%%%%%%%%% PROOF PROOF PROOF
%%%%%%%%%%%%%%%%%%%%%%%%%%%%%% PROOF PROOF PROOF

%%%%%%%%%%%%%%%%%%%%%%%%%%%%%% PROOF PROOF PROOF
%%%%%%%%%%%%%%%%%%%%%%%%%%%%%% PROOF PROOF PROOF
\begin{proof}[Proof of Theorem \ref{t.singleradiusN}] Given $t\geq 1$ let $N\geq B t^2$ be a positive integer for some numerical constant $B>0$ to be determined later. By corollary \ref{c.tor2t} we have that
	$$D_s(f_N,2)\gtrsim \sqrt{s},$$
where $s$ is equal to either $t$ or $2t$. Consider the cube $Q_1\eqdef[s,N-s]^2$. We have
\begin{align*}
	\int_{[s,N-s]^2}|D_s(f_N)(x)|^2dx &=D_s(f_N,2)^2-\frac{1}{N^2}\int_{[-s,N+s]^2\setminus [s,N-s]^2 } |f_N*d\sigma_s(x)|^2dx \\ 
    &\gtrsim s(1-24 s^2/N )\gtrsim s,
\end{align*}
if $B$ is large enough, say $B\geq 100$. Since all the circles with centers in $[s,N-s]^2$ and radius $s$ are contained in $Q_N$ this proves Theorem \ref{t.fullcircleN}.
\end{proof}
%%%%%%%%%%%%%%%%%%%%%%%%%%%%%% PROOF PROOF PROOF
%%%%%%%%%%%%%%%%%%%%%%%%%%%%%% PROOF PROOF PROOF

%%%%%%%%%%%%%%%%%%%%%%%%%%%%%% SECTION  SECTION SECTION
%%%%%%%%%%%%%%%%%%%%%%%%%%%%%% SECTION  SECTION SECTION 
\section{Single radius discrepancy for arcs} \label{s.singleradius}
Theorem \ref{t.fullcircleN} solves the problem of finding a full circle with large discrepancy. There is one element however that is not very satisfactory, namely the fact that we cannot guarantee that for \emph{every} radius $t\geq 1$ there corresponds a circle of radius \emph{exactly} $t$ with large discrepancy. The problem is caused by the roots of $\hat\sigma_1(\xi )$ which allow the expression
$$\int_{\mathbb R^2} |\hat f(\xi)|^2 |\hat \sigma_t (\xi)|^2d\xi$$
to become small. When $N\simeq t$ we can deal with this problem by essentially throwing away small neighborhoods of the roots of $\hat \sigma_1$ and showing that we don't loose much of the $L^2$ mass of the function $\hat f$.

We begin by analyzing the behavior of $\hat \sigma_1(|\xi|)$. By standard estimates we have the asymptotic expansion
\begin{equation}\label{e.asymptotic}
	\hat\sigma_1(\xi)=\hat\sigma_1(|\xi|)=2|\xi|^{-\frac{1}{2}}\cos\big(2\pi|\xi|-\frac{\pi}{4}\big)+O(|\xi|^{-\frac{3}{2}}),\quad |\xi|\to+\infty;	
\end{equation}
see for example \cite{W}. Observe that the cosine term in the asymptotic formula above vanishes exactly when
$$|\xi|	=\beta_k\eqdef \big(\frac{k}{2}+\frac{3}{8}\big), \quad k=0,1,2,\ldots.$$
For a small parameter $0<w<\frac{1}{8}$ we define the neighborhoods 
$$A_w(\beta_k)\eqdef \{\xi\in\mathbb R^2: | |\xi|-\beta_k| < w\}.$$
Observe that our choice of $w$ implies that the $A_w$'s do not overlap. The following lemma analyzes the behavior of $\hat \sigma_1$ away from the annuli $A_w$

%%%%%%%%%%%%%%%%%%%%%%%%%%%%%% LEMMA LEMMA LEMMA
%%%%%%%%%%%%%%%%%%%%%%%%%%%%%% LEMMA LEMMA LEMMA
\begin{lemma}\label{l.lowerestimate} For every sufficiently small $w>0$ there exists a constant $c(w)$ such that
	$$|\hat \sigma_1(|\xi|)|^2 \gtrsim_w \begin{cases} \frac{1}{|\xi|}, &\quad |\xi|>c(w), \quad \xi\notin \cup_k A_w(\beta_k) \\
		 1 ,&\quad |\xi|\leq c(w),\quad \xi\notin \cup_k A_w(\gamma_k),
	\end{cases}$$
where $\gamma_1<\gamma_2<\ldots<\gamma_M$ are the roots of $\hat \sigma_1$ in $\{|\xi|<c(w)\}$.
\end{lemma}
%%%%%%%%%%%%%%%%%%%%%%%%%%%%%% LEMMA LEMMA LEMMA
%%%%%%%%%%%%%%%%%%%%%%%%%%%%%% LEMMA LEMMA LEMMA

%%%%%%%%%%%%%%%%%%%%%%%%%%%%%% PROOF PROOF PROOF
%%%%%%%%%%%%%%%%%%%%%%%%%%%%%% PROOF PROOF PROOF
\begin{proof} 
	By \eqref{e.asymptotic} there exist constants $c_1,c_2>0$ such that for $|\xi|>c_1$ we have

	\begin{align*}
		|\hat \sigma_1 (|\xi|)|^2&\gtrsim \frac{1}{|\xi|}\bigg( \Abs{\cos(2\pi |\xi|-\frac{\pi}{4})} ^2 -\frac{c_2}{|\xi|} \bigg).
	\end{align*}
Now the minimum of the cosine term in the region $\{|\xi|>c_1\}\setminus \cup_k A_w(\beta_k) $ is obviously achieved when $||\xi|-\beta_k|=w$ for some $k$.	If $w<\frac{1}{4}$ we have
$$\Abs{\cos(2\pi |\xi|-\frac{\pi}{4})}\geq 4 \Abs{|\xi|-\beta_k}=4w$$
We can thus estimate
	\begin{align*}
	|\hat \sigma_1(|\xi|)|^2 & \gtrsim \frac{1}{|\xi|} \bigg (16w^2 -\frac{c_2}{|\xi|}\bigg) \gtrsim_w \frac{1}{|\xi| },
	\end{align*}
	whenever $|\xi|>\frac{c_2}{8w^2}\eqdef c(w)$ and $\xi\notin \cup_k A_w(\beta_k)$.
	
Now there are finitely many roots of $\hat \sigma_1(\xi)$ in the ball $\{|\xi|\leq c(w)\}$ and let us denote them by $\gamma_1<\gamma_2<\cdots<\gamma_M$. By compactness we have that $|\hat \sigma_1(\xi)|^2 \gtrsim_w 1$ whenever $|\xi|\leq c(w)$ and $x\notin \cup_k A_w(\gamma_k)$. In order to make sure that all the annuli are non-overlapping we have to take $w<\min\{\frac{1}{8},\frac{1}{2}\min_k(\gamma_{k+1}-\gamma_k),\beta_0-\gamma_M\}\eqdef w_0$.
 \end{proof}
%%%%%%%%%%%%%%%%%%%%%%%%%%%%%% PROOF PROOF PROOF
%%%%%%%%%%%%%%%%%%%%%%%%%%%%%% PROOF PROOF PROOF

Lemma \ref{l.lowerestimate} can be used to obtain a favorable estimate for $D_t(f,2)$ as follows. Adopting the notations of Lemma \ref{l.lowerestimate} and invoking Plancherel's theorem we write for every $w<w_o$ small enough (remember $t\simeq N$)
	\begin{align*}
	D_t(f_N,2)^2  &\gtrsim  \frac{1}{N^2} \int_{ \{|\xi|<c(w) \}\setminus \cup_k A_w(\gamma_k) } |\widehat {f_N}(\xi/t)|^2 |\hat \sigma_1 (|\xi|)|^2d\xi + \frac{1}{N^2}\int_{\{c(w)<|\xi|<t)\}\setminus \cup_k A_w(\beta_k)} |\widehat {f_N}(\xi/t)|^2 |\hat\sigma_1(|\xi|)|^2d\xi \\
&\gtrsim_w \frac{1}{N^2} \int_{ \{|\xi|<c(w) \}\setminus \cup_k A_w(\gamma_k) }  |\widehat{f_N}(\xi/t)|^2 d\xi +  \frac{1}{N^2}\int_{\{c(w)<|\xi|<t)\}\setminus \cup_k A_w(\beta_k)}   |\widehat {f_N}(\xi/t)|^2 \frac{1}{|\xi|}d\xi.
\end{align*}
Setting $E_w\eqdef ( \cup_k A_w(\gamma_k)) \cup (  \cup_k A_w(\beta_k)  )$ and combining the previous estimates we have
\begin{align}\label{e.estimatewithholes}
D_t(f_N,2)^2&\gtrsim_w \frac{1}{tN^2} \int_{\{|\xi|<t \}\setminus E_w}  |\widehat {f_N}(\xi/t)|^2 d\xi =  \frac{t}{N^2}\int_{B(0,1)\setminus \frac{1}{t}E_w} |\widehat {f_N}(\xi)|^2 d\xi .
\end{align}

The following Poincar\'e-type inequality will allow us to show that the $L^2$ norm of $\widehat{f_N}$ on $B(0,1)\setminus \frac{1}{t} E_w$ is comparable to the the full $L^2$ norm of $\widehat {f_N}$.	
	
%%%%%%%%%%%%%%%%%%%%%%%%%%%%%% PROPOSITION PROPOSITION PROPOSITION
%%%%%%%%%%%%%%%%%%%%%%%%%%%%%% PROPOSITION PROPOSITION PROPOSITION
\begin{proposition}\label{p.poincare}
	For any positive integer $d\geq 1$ let $B=B(0,R)\subset \mathbb R^d$ be a Euclidean ball in the $d$-dimensional Euclidean space, centered at the origin, and $g\in  C^1(B)$. Suppose that  $0<\beta_1<\beta_2<\cdots<\beta_N<R$. We set $\beta_0\eqdef 0$ and $\beta_{N+1}\eqdef R$ and
	$$\beta\eqdef \min_{1\leq n\leq N+1 }(\beta_{n}-\beta_{n-1}).$$
For $k=1,2,\ldots,N$,		 we set
$$A_w(\beta_k)\eqdef \{\xi\in\mathbb R^d: | |\xi|-\beta_k| < w\}.$$
Then for $0<w<\beta/3$ we have that
		$$ \int_{ B}|g(x)|^2 \lesssim \int_{ B\setminus (\cup_{n=1} ^N A_w(\beta_n))  }|g(x)|^2+ w^2\int_{ B} |\nabla g(x)|^2 .$$
\end{proposition}
%%%%%%%%%%%%%%%%%%%%%%%%%%%%%% PROPOSITION PROPOSITION PROPOSITION
%%%%%%%%%%%%%%%%%%%%%%%%%%%%%% PROPOSITION PROPOSITION PROPOSITION

%%%%%%%%%%%%%%%%%%%%%%%%%%%%%% PROOF PROOF PROOF
%%%%%%%%%%%%%%%%%%%%%%%%%%%%%% PROOF PROOF PROOF
\begin{proof}
	We first focus on a single annulus $A_s(\beta_n)$ for some $1\leq n \leq N$ and some real parameter $s$ in the interval $[w,2w)$. For $\beta_n-s<r\leq \beta_n+s$ and $u\in S^{d-1}$ we have that
	$$g(ru) = g((\beta_n-s)u)+ \int_{\beta_n-s} ^r \partial_t( g(tu)) dt.$$	

	Using the simple inequality $\frac{1}{2}(a+b)^2\leq a^2+b^2$ for $a,b\in\mathbb R$, and the Cauchy-Schwarz inequality we conclude that 
	$$|g(ru)|^2 \lesssim |g((\beta_n-s)u)|^2+ 2s \int_{\beta_n-s} ^{\beta_n+s} |\partial_t( g(tu))|^2 dt.$$

	Multiplying by $r^{d-1}$ and integrating for $r\in[\beta_n-s,\beta_n+s)$ and $u\in S^{d-1}$, we get
	\begin{align}
		\int_{A_s(\beta_n)} |g(x)|^2dx &\lesssim \int_{\beta_n-s} ^{\beta_n+s} r^{d-1}dr \int_{S^{d-1}}|g((\beta_n-s)u)|^2 d\sigma^{(d-1)} _1 (u)\\
		&+ 2s \int_{S^{d-1}}\bigg(\int_{\beta_n-s} ^{\beta_n+s}\Big(\int_{\beta_n-s} ^{\beta_n+s} |\partial_t( g(tu))|^2 dt\Big) r^{d-1}dr\bigg) d\sigma_1 ^{(d-1)}(u).
	\end{align}
	Now observe that for $w<s\leq 2w$ and $r,t\in[\beta_n-s,\beta_n+s)$, we have that $r\simeq t \simeq \beta_n$. Hence,
	$$\int_{A_s(\beta_n)} |g(x)|^2dx \lesssim s \beta_n ^{d-1}\int_{S^1}|g((\beta_n-s)u)|^2 du+ s^2 \int_{A_s (\beta_n)} |\nabla  g(x) |^2dx.$$
	Integrating the left hand side for $s\in[w,2w)$ we see that
	$$\int_w ^{2w} \int_{A_s(\beta_n)}|g(x)|^2dx ds \gtrsim w \int_{A_w(\beta_n)}|g(x)|^2dx.$$
	On the other hand
	\begin{align}  \beta_n ^{d-1} \int_w ^{2w}  \int_{S^1}|g((\beta_n-s)u)|^2 s du  ds & \lesssim w \int _{B(0,\beta_n-w) \setminus \ B(0,\beta_n-2w)}|g(x)|^2 dx
		\\ &\leq w\int_{B(0,\beta_n-w) \setminus \ B(0,\beta_{n-1}+w)}|g(x)|^2dx,
	\end{align}
	since $w<\beta/3$. Finally we readily see that
	$$\int_w ^{2w} s^2 \int_{A_s(\beta_n) } |\nabla g(x)|^2 dx ds \lesssim  w^3 \int_{A_{2w}(\beta_n)} |\nabla g(x)|^2 dx.$$
	Putting these estimates together get
	$$ w \int_{A_w(\beta_n)}|g(x)|^2dx \lesssim  w\int_{B(0,\beta_n-w) \setminus \ B(0,\beta_{n-1}+w)}|g(x)|^2dx +w^3 \int_{A_{2w}(\beta_n)} |\nabla g(x)|^2 dx.$$
Observe that we have
	 $$\cup_{n=1} ^N {B(0,\beta_n-w) \setminus \ B(0,\beta_{n-1}+w)} \subseteq   B\setminus \big(\cup_{n=1} ^N A_w (\beta_n) \big)$$
	and the unions on both sides of the inclusion above are disjoint. Summing in $n$ we thus get
		$$\int_{\cup_{n=1} ^N  A_w(\beta_n)}|g(x)|^2dx \lesssim  \int_{B\setminus \cup_{n=1} ^N A_w (\beta_n) }|g(x)|^2dx +w^2 \int_{\cup_{n=1} ^N A_{2w}(\beta_n)}|\nabla g(x)|^2 dx.$$
	Adding the term $\int_{B\setminus  \cup_{n=1} ^N A_w (\beta_n) }|g(x)|^2dx$ in both sides of the inequality completes the proof.\end{proof}	
%%%%%%%%%%%%%%%%%%%%%%%%%%%%%% PROOF PROOF PROOF
%%%%%%%%%%%%%%%%%%%%%%%%%%%%%% PROOF PROOF PROOF

Now Proposition \ref{p.poincare} and estimate \eqref{e.estimatewithholes} will allow us to conclude the proof of Theorem \ref{t.singleradiusN}:

%%%%%%%%%%%%%%%%%%%%%%%%%%%%%% PROOF PROOF PROOF
%%%%%%%%%%%%%%%%%%%%%%%%%%%%%% PROOF PROOF PROOF
\begin{proof}[Proof of Theorem \ref{t.singleradiusN}] Estimate \eqref{e.estimatewithholes} and Proposition \ref{p.poincare} imply that
\begin{align}\label{e.break}
	D^2 _t(f_N,2) & =\frac{1}{t}\int_{B(0,1)\setminus \frac{1}{t}E_w} |\widehat {f_N}(\xi)|^2 d\xi   \gtrsim_w \frac{1}{t}\bigg( \int_{B(0,1)}|\widehat{f_N}(\xi)|^2d\xi -\frac{w^2}{t^2} \int_{B(0,1)}|\nabla{\widehat{f_N}}(\xi)|^2d\xi\bigg).
\end{align}
Using the bounds
$$\int_{B(0,1)}|\widehat f_N(\xi)|^2d\xi\gtrsim\int_{[-\frac{1}{2},\frac{1}{2}]^2}\bigg|\sum_{j,k=0} ^{N-1}z_{jk} e^{2\pi i (j\xi_1+k\xi_2)}\bigg|^2d\xi\geq N^2,$$
and
$$ \int_{B(0,1) }|\nabla{\widehat{f_N}}(\xi)|^2d\xi\leq \int_{\mathbb R^2} |x|^2|f_N(x)|^2dx\lesssim N^4$$
in estimate \eqref{e.break} we get
\begin{align*}
	D^2 _t(f_N,2)	\gtrsim\frac{1}{t} \big(N^2-\frac{w^2}{t^2} N^4\big)\simeq t (1-c w^2),
\end{align*}
for some constant $c>0$. If $w$ is sufficiently small we conclude that $D_t(f_N,2)\gtrsim_w \sqrt{t}$ as we wanted to prove. 
\end{proof}
%%%%%%%%%%%%%%%%%%%%%%%%%%%%%% PROOF PROOF PROOF
%%%%%%%%%%%%%%%%%%%%%%%%%%%%%% PROOF PROOF PROOF

%%%%%%%%%%%%%%%%%%%%%%%%%%%%%% REMARK REMARK REMARK
%%%%%%%%%%%%%%%%%%%%%%%%%%%%%% REMARK REMARK REMARK
\begin{remark} The calculations in this section show that
		$$D_t(f_N,2)\gtrsim \sqrt{t},$$
for $N\simeq t$. This alone is not enough to conclude the existence of a \emph{full circle} with large discrepancy $\sim \sqrt{t}$. Indeed, the argument used in the proof of Theorem \ref{t.fullcircleN} requires the validity of the previous estimate for $N\gtrsim t^2$ while, here, we only have it for $N\simeq t$.
\end{remark}
%%%%%%%%%%%%%%%%%%%%%%%%%%%%%% REMARK REMARK REMARK
%%%%%%%%%%%%%%%%%%%%%%%%%%%%%% REMARK REMARK REMARK

%%%%%%%%%%%%%%%%%%%%%%%%%%%%%% SECTION  SECTION SECTION
%%%%%%%%%%%%%%%%%%%%%%%%%%%%%% SECTION  SECTION SECTION
\section{Discrepancy with respect to general sets}\label{s.generalsets}
In this section we study the discrepancy of a coloring $f$ of the plane with respect to more general families of sets. To keep the exposition relatively simple let us assume that $S$ is a simple, closed, piecewise $C^1$ curve in the Euclidean plane and let $K$ denote its interior. Let $d\sigma_S$ denote the arc-length measure on $S$. In the previous sections we have studied the discrepancy of $f$ with respect to the family of all dilations and translations of the unit circle. Here, the relevant families are
$$\{x+r\tau K:x\in\mathbb R^2,\ r>0,\ \tau\in SO(2)\},$$
and
$$\{x+r\tau S:x\in\mathbb R^2,\ r>0,\ \tau\in SO(2)\}.$$
Note that we introduce rotations which was superfluous in the case of the circle. Here however it is absolutely essential. Indeed, consider the standard chessboard-like alternating coloring (i.e. adjacent squares have different colors) and let $K$ be the unit square with its sides parallel to the coordinate axes. Obviously the discrepancy of this coloring with respect to the dilations and translations of $K$ (or $\partial K$) is $\sim 1$ so the problem is trivial. Another option would be to place certain assumptions on the curvature of $\partial K$ but we will not pursue this here.

For $x\in\mathbb R^2,r>0$ and $\tau \in SO(2)$ we define
$$D_K(f_N,x,\tau,r)\eqdef (f*\chi_{r\tau K})(x)$$
and
$$D_S(f_N,x,\tau,r)\eqdef (f*d\sigma_{r\tau S})(x).$$

%%%%%%%%%%%%%%%%%%%%%%%%%%%%%% SECTION  SECTION SECTION
%%%%%%%%%%%%%%%%%%%%%%%%%%%%%% SECTION  SECTION SECTION
\subsection{Average estimates for the Fourier transform.}\label{s.lowerfourier}
We will obtain lower bounds on the discrepancies described above by studying their $L^2 $ averages. The most important ingredient of this approach is the following lemma describing the average asymptotic behavior of the Fourier transform of $d\sigma_S$ and $\chi_K$. These estimates are essentially contained in the proof of \cite[Theorem 3, Chapter 6]{M}.
%%%%%%%%%%%%%%%%%%%%%%%%%%%%%% LEMMA LEMMA LEMMA
%%%%%%%%%%%%%%%%%%%%%%%%%%%%%% LEMMA LEMMA LEMMA
\begin{lemma}\label{l.fourierKS} Let $S$ be a simple, closed, piecewise $C^1$ curve in the Euclidean plane and denote by $K$ its interior so that $S=\partial K$. There exist numerical constants $A_o>1$ and $R_o>0$ such that, if $R>R_o$ and $A>A_o$, then
$$\int_{R\leq |\xi|\leq AR} |\widehat{\chi_K}(\xi)|^2d\xi \gtrsim_A \frac{|S|}{R},$$
and 
$$\int_{R\leq |\xi|\leq AR} |\widehat{d\sigma_S}(\xi)|^2 d\xi \gtrsim_A |S| R.$$
Here $|S|$ denotes the arc-length of $S$.
\end{lemma}
%%%%%%%%%%%%%%%%%%%%%%%%%%%%%% LEMMA LEMMA LEMMA
%%%%%%%%%%%%%%%%%%%%%%%%%%%%%% LEMMA LEMMA LEMMA

%%%%%%%%%%%%%%%%%%%%%%%%%%%%%% PROOF PROOF PROOF
%%%%%%%%%%%%%%%%%%%%%%%%%%%%%% PROOF PROOF PROOF
\begin{proof}We follow Montgomery from \cite[Theorem 3, Chapter 6]{M}. For $r>0$ we set
	$$g(r)=\int_{S^1}|\widehat{\chi_K}(r\xi')|^2d\sigma_1(\xi').$$
Under our assumptions on $K$, Montgomery proves the asymptotic estimate
$$\int_0 ^R g(r)r^5 dr\simeq |S|R^3,$$
as $R\to +\infty$. 	This means that there exist numerical constants $R_o,c_1,c_2>0$ such that
$$c_1 |S| R^3 \leq \int_0 ^R g(r)r^5 dr \leq c_2|S|R^3,$$
whenever $R>R_o$. For $A>1$ and $R>R_o$ we thus have

\begin{align*}
	\int_{R} ^{AR} g(r) r^5 dr \geq |S| R^3 (A^3c_1 -c_2) \gtrsim |S| R^3
\end{align*}
if $A>A_0$ where $A_o>1$ is a numerical constant. We conclude that
\begin{align}
 \int_{R\leq |\xi|\leq AR}  |\widehat{\chi_K}(\xi)|^2 d\xi= \int_R ^{AR} g(r) r dr \gtrsim_A |S|/R,
\end{align}
whenever $R>R_o$ and $A>A_o$. This proves the first estimate of the lemma. 

For the second we modify the proof of \cite[Theorem 3, Chapter 6]{M}. With $h(x)=e^{-\pi R^2|x|^2}$ Montgomery shows that $ \|\chi_K*\nabla^2 h\|_2 ^2\simeq |S| / R  $ as $R\to +\infty$. On the other hand, by Green's theorem we have
$$\chi_K*\nabla^2 h=\int_S \frac{\partial h}{\partial\textbf n}(x-y) d\sigma_S(y).$$
Combining these two facts and using Plancherel's theorem we get
\begin{align*}
\frac{|S|}{R}&\simeq \int_{\mathbb R^2}\big| \widehat{\frac{\partial h}{\partial\textbf n}}(\xi)\big|^2	|\widehat{d\sigma_S}(\xi)|^2d\xi=\int_{\mathbb R^2}|\xi\cdot n|^2|\hat h(\xi)|^2 |\widehat{d\sigma_S}(\xi)|^2d\xi \\
&= \int_0 ^{+\infty}  |\hat h(r)|^2 \bigg( \int_{S^1} |\xi'  \cdot n|^2 |\widehat {d\sigma_S}( r \xi')|^2d\sigma(\xi') \bigg)\, r^3 dr,
\end{align*}
where $\hat h(r)=\hat h(|\xi|)=R^{-2}e^{-\pi R^{-2}r^2}$. Let us call
\begin{align}\label{e.y(r)}
	y(r)\eqdef  \int_{S^1} |\xi ' \cdot n|^2 |\widehat{d\sigma_S}( r \xi')|^2d\sigma(\xi') \leq \int_{S^1}  |\widehat{d\sigma_S}( r \xi')|^2d\sigma(\xi')  .
\end{align}
We have
$$\int_0 ^{+\infty} y(r) e^{-2\pi r^2/R^2} r^3 dr \simeq |S| R^3.$$
As in Montgomery \cite{M}, we use the Hardy-Littlewood Tauberian theorem \cite[Theorem 108]{HL} to conclude that
$$\int_0 ^R y(r) r^3 dr\simeq |S| R^3, $$
as $R\to+\infty$. Arguing as in the first part of the proof we conclude that there exist numerical constants $R_o$ and $A_o>1$ such that
$$\int_{R} ^{AR} y(r) r^3dr \gtrsim  |S| R^3,$$
whenever $R>R_o$ and $A>A_o$. By \eqref{e.y(r)} we conclude that
\begin{align*}
	\int_{R\leq |\xi|\leq AR} |\widehat{d\sigma_S}(\xi)|^2 d\xi  \geq \int_R ^{AR} y(r)r dr\gtrsim_A|S|R,
\end{align*}
as we wanted to prove.
\end{proof}
%%%%%%%%%%%%%%%%%%%%%%%%%%%%%% PROOF PROOF PROOF
%%%%%%%%%%%%%%%%%%%%%%%%%%%%%% PROOF PROOF PROOF

%%%%%%%%%%%%%%%%%%%%%%%%%%%%%% SECTION  SECTION SECTION
%%%%%%%%%%%%%%%%%%%%%%%%%%%%%% SECTION  SECTION SECTION
\subsection{Lower bounds for discrepancy with respect to general sets} Using the average estimates for the Fourier transform of $\chi_K$ and $d\sigma_S$ proved in the previous paragraph we can now show the desired lower bounds for the (average) discrepancy.

%%%%%%%%%%%%%%%%%%%%%%%%%%%%%% THEOREM THEOREM THEOREM
%%%%%%%%%%%%%%%%%%%%%%%%%%%%%% THEOREM THEOREM THEOREM
\begin{theorem}\label{t.segment} Let $S$ be a simple, closed, piecewise $C^1$ curve and denote by $K$ its interior.
	\begin{list}{}{}
		\item[(i)] For every positive integer $N$ there exists a $x\in Q_N$, a dilation $r\simeq N$ and a rotation $\tau\in SO(2)$ such that
	$$D_K (f_N,x,r,\tau)  \gtrsim_K \sqrt{N},$$
where the implied constant depends only on $K$. 
\item [(ii)]  For every positive integer $N$ there exists a $x\in Q_N$, a dilation $r\simeq N$ and a rotation $\tau\in SO(2)$ such that
$$D_S (f_N,x,r,\tau)  \gtrsim_K \sqrt{N},$$
where the implied constant depends only on $S=\partial K$. 
\end{list}
\end{theorem}
%%%%%%%%%%%%%%%%%%%%%%%%%%%%%% THEOREM THEOREM THEOREM
%%%%%%%%%%%%%%%%%%%%%%%%%%%%%% THEOREM THEOREM THEOREM

%%%%%%%%%%%%%%%%%%%%%%%%%%%%%% REMARK REMARK REMARK
%%%%%%%%%%%%%%%%%%%%%%%%%%%%%% REMARK REMARK REMARK
\begin{remark} As in Theorem \ref{t.singleradiusN} we cannot guarantee that the sets $x+r\tau K$, $x+r\tau S$ of the previous theorem are fully contained in $Q_N$. Thus, Theorem \ref{t.segment} only implies the existence of \emph{a segment} of $K$ or $S$ which has large discrepancy with respect the coloring of the whole plane $f$.	
\end{remark}
%%%%%%%%%%%%%%%%%%%%%%%%%%%%%% REMARK REMARK REMARK
%%%%%%%%%%%%%%%%%%%%%%%%%%%%%% REMARK REMARK REMARK

In order to prove Theorem \ref{t.segment} we will consider the average discrepancy
$$D_K(f_N,2)^2\eqdef\frac{1}{N^3}\int_{SO(2)}\int_{a N} ^{\beta N}\bigg(  \int_{\mathbb R^2}D_{r\tau K}(f_N,x)^2 dx\bigg) \ dr d \tau ,$$
where $0<a<\beta$ will be appropriate numerical constants.
Similarly define
$$D_S(f_N,2)^2\eqdef\frac{1}{N^3}\int_{SO(2)}\int_{aN} ^{\beta N}\bigg(  \int_{\mathbb R^2}D_{r\tau S}(f_N,x)^2 dx\bigg) \ dr d  \tau .$$
The factor $1/N^3$ is there to almost normalize the measure while $d\tau$ is the normalized Haar measure on $SO(2)$.
%%%%%%%%%%%%%%%%%%%%%%%%%%%%%% PROOF PROOF PROOF
%%%%%%%%%%%%%%%%%%%%%%%%%%%%%% PROOF PROOF PROOF
\begin{proof} The proofs of \textit{(i)} and \textit{(ii)} are essentially identical so we will just prove \textit{(ii)}.
Using Plancherel's theorem we have
\begin{align*}
	D _S(f_N,2)^2 & =\frac{1}{N^3} \int_{\mathbb R^2} \Abs{\widehat{f_N}(\xi)}^2 \bigg( \int_{a N} ^{\beta N} r^2 \int_{SO(2)} \Abs{\widehat{d\sigma_S}(r \tau^{-1} \xi)}^2 d\tau dr \bigg) \ d\xi\\ 
	& = \frac{1}{N^3} \int_{\mathbb R^2} \Abs{\widehat{f_N}(\xi)}^2 \bigg(\int_{a |\xi| N} ^{\beta |\xi| N } \frac{r^2}{|\xi|^3} \int_{SO(2)} \Abs{\widehat{d\sigma_S}(r \tau e_\xi )}^2 d \tau dr\bigg)\  d\xi\\
	&\gtrsim \frac{1}{N^2}\int_{|\xi|<1}\Abs{\widehat{f_N}(\xi)}^2 \frac{1}{|\xi|^2} \bigg(\int_{a |\xi| N} ^{\beta |\xi| N }  \int_{SO(2)} \Abs{\widehat{d\sigma_S}(r \tau e_\xi )}^2 r d \tau  dr\bigg)\ d\xi \\
	&=\frac{1}{N^2}  \int_{|\xi|<1}\Abs{\widehat{f_N}(\xi)}^2  \frac{1}{|\xi|^2} \int_{\{ a|\xi|N<|y|<\beta |\xi|N \}}\Abs{\widehat{d\sigma_S}(y)}^2 dy\  d\xi\\ 
	&= \frac{1}{N^2} \int_{\{a|\xi|N< M\}} \Abs{\widehat{f_N}(\xi)}^2 \frac{1}{|\xi|^2}\bigg(  \int_{\{ a|\xi|N<|y|<\beta |\xi|N \}}\Abs{\widehat{d\sigma_S}(y)}^2 \ \bigg) d\xi \\ 
	& \quad +\frac{1}{N^2}\int_{\{M<a|\xi|N<a N\}} \Abs{\widehat{f_N}(\xi)}^2\frac{1}{|\xi|^2}   \int_{\{ a|\xi|N<|y|<\beta |\xi|N \}}\Abs{\widehat{d\sigma_S}(y)}^2 dy\  d\xi\eqdef I+II.
\end{align*}
Here we have set $e_\xi= \frac{\xi}{|\xi|}$. Using Lemma \ref{l.fourierKS} we get for $M>R_o$ and $\beta/a>A_o$, that
$$II\gtrsim \frac{1}{N^2} \int_{\{M<a|\xi|N<a N \}}  \Abs{\widehat{f_N}(\xi)}^2 \frac{N|S|}{|\xi|}d\xi\geq \frac{|S|}{N} \int_{\{\frac{M}{a N}< |\xi| <1 \}}  \Abs{\widehat{f_N}(\xi)}^2d\xi.$$

Now for small $\epsilon>0$ we write
\begin{align*}
	I& \geq \int_{ \{ \frac{\epsilon}{N}<|\xi|<\frac{M}{a N} \}}  \Abs{\widehat{f_N}(\xi)}^2  \int_{\{ a|\xi|N<|y|<\beta |\xi|N \}}\Abs{\widehat{d\sigma_S}(y)}^2 dy\  d\xi \\
	&\gtrsim_S \int_{ \{ \frac{\epsilon}{N}<|\xi|<\frac{M}{a N} \}}  \Abs{\widehat{f_N}(\xi)}^2  d\xi.
	\end{align*} 
The last estimate is justified since the region $\{a|\xi|N<|y|<\beta|\xi|N\}$ is an annulus inside $B (0,M\beta/a )$, of width at least $(\beta-a)\epsilon$, and $\widehat{d\sigma_S}(y )$ does not vanish identically on any annulus. Adding the estimates we conclude 

\begin{align*}
	D _S(f_N,2)^2 & \gtrsim_S \frac{1}{N }\int_{\{\frac{\epsilon}{N}<|\xi|<1\}} \Abs{\widehat{f_N}(\xi)}^2 d\xi 
	\geq \frac{1}{N}\bigg(\int_{[-\frac{1}{2},\frac{1}{2}]^2}\Abs{\widehat{f_N}(\xi)}^2 d\xi - \int_{|\xi|<\frac{\epsilon}{N}}\Abs{\widehat{f_N}(\xi)}^2 d\xi  \bigg).
\end{align*}
Now using the trivial bound $\|\widehat{f_N}\|_{L^\infty(\mathbb R^2)} ^2 \leq \| f_N\|_{L^1(\mathbb R ^2)} ^2=N^4$ and \eqref{e.1period} we get
\begin{align*}
	D _S(f_N,2)^2  \gtrsim \frac{1}{N}(N^2-N^4\frac{\epsilon^2}{N^2})\gtrsim  N,
\end{align*}
if $\epsilon$ is small enough.
\end{proof}
%%%%%%%%%%%%%%%%%%%%%%%%%%%%%% PROOF PROOF PROOF
%%%%%%%%%%%%%%%%%%%%%%%%%%%%%% PROOF PROOF PROOF

%%%%%%%%%%%%%%%%%%%%%%%%%%%%%% REMARK REMARK REMARK
%%%%%%%%%%%%%%%%%%%%%%%%%%%%%% REMARK REMARK REMARK
\begin{remark}
	By using the same ideas as in the proof of Theorem \ref{t.fullcircleN} we can show a stronger result in the special case of the Euclidean ball. In particular, we have that for every checkerboard coloring $f$ of the whole plane and every $t\geq 1$, there is a $x\in\mathbb R^2$ such that 
	$$ \mbox{either}\quad\bigg|\int_{B(x,t)} f(y)dy\bigg|\gtrsim \sqrt{t}\quad\mbox{or}\quad \bigg| \int_{B(x,2t)} f(y)dy\bigg|\gtrsim\sqrt{2t}.$$
\end{remark}
\begin{remark} The only limitation in the choice of the set $K$ and the curve $S$ come from Lemma \ref{l.fourierKS}. Going back to Montgomery's proof in \cite{M} one see that Lemma \ref{l.fourierKS} remains valid if $K$ is for example a multiply connected set and $S$ is replaced by $\partial K$. Furthermore, the $C^1$ condition of the boundary can be replaced by the weaker condition that the limit
	$$\lim_{\delta\to 0}\frac{|\{x\in\mathbb R^2: \textnormal{dist}(x,S)<\delta\}|}{\delta},$$
	exists and is finite.
\end{remark}
%%%%%%%%%%%%%%%%%%%%%%%%%%%%%% REMARK REMARK REMARK
%%%%%%%%%%%%%%%%%%%%%%%%%%%%%% REMARK REMARK REMARK

%%%%%%%%%%%%%%%%%%%%%%%%%%%%%% SECTION  SECTION SECTION
%%%%%%%%%%%%%%%%%%%%%%%%%%%%%% SECTION  SECTION SECTION

\end{document}